\documentclass{llncs}
\usepackage{amsfonts}
\usepackage{amssymb}
\usepackage{amsmath}
\usepackage{graphicx}
\usepackage{psfrag}
\usepackage{xcolor}
\usepackage{enumerate}
\usepackage{lineno}
\usepackage{mathtools}
\usepackage{tikz-cd}
\usepackage[title]{appendix}
\usetikzlibrary{arrows}
\usepackage[all]{xy}

\newcommand{\R}{{\mathbb{R}}}
\newcommand{\Rl}{{(\mathbb{R}, \le)}}
\newcommand{\G}{\mathbf{Graph}}
\newcommand{\GM}{\mathbf{Graph}_{\textnormal{m}}}
\renewcommand{\S}{\mathbf{Set}}
\renewcommand{\P}{\mathbf{Poset}}
\newcommand{\PM}{\mathbf{Poset}_{\textnormal{m}}}
\newcommand{\WP}{\mathbf{WDPoset}}
\newcommand{\WPM}{\mathbf{WDPoset}_{\textnormal{m}}}
\newcommand{\C}{\mathbf{C}}
\newcommand{\CM}{\mathbf{C}_{\textnormal{m}}}
\newcommand{\Obj}{{\textnormal{Obj}}}

\newcommand{\id}{{\textnormal{Id}}}
\newcommand{\core}{{\textnormal{core}}}
\newcommand{\free}{{\textnormal{Free}}}
\newcommand{\maximal}{{\textnormal{M}}}

\begin{document}

\title{Beyond topological persistence: Starting from networks}

%\amsclass{68R10,05C10,18C99}
\author{Mattia G. Bergomi$^1$, Massimo Ferri$^2$, Pietro Vertechi$^3$, Lorenzo Zuffi$^2$}

\authorrunning{MG Bergomi, M Ferri, P Vertechi, L Zuffi}

\institute{$^1$ Veos Digital, Milan, Italy\\
$^2$ ARCES and Dept. of Mathematics, Univ. of Bologna, Italy\\
$^3$ Champalimaud Center for the Unknown, Lisbon, Portugal
\email{mattia.bergomi@veos.digital, massimo.ferri@unibo.it, pietro.vertechi@neuro.fchampalimaud.org, lorenzo.zuffi@studio.unibo.it},
}
%
%
% \institute{M.G. Bergomi, P. Vertechi \at Champalimaud Center for the Unknown, Lisbon, Portugal\\
% \email{mattia.bergomi@neuro.fchampalimaud.org, pietro.vertechi@neuro.fchampalimaud.org}\\
% M. Ferri, L. Zuffi \at ARCES and Dept. of Mathematics, Univ. of Bologna, Italy\\
% \email{massimo.ferri@unibo.it, lorenzo.zuffi@studio.unibo.it},
% }

\maketitle              % typeset the title of the contribution
\begin{abstract}
Persistent homology enables fast and computable comparison of topological objects. However, it is naturally limited to the analysis of topological spaces. We extend the theory of persistence, by guaranteeing robustness and computability to significant data types as simple graphs and quivers. We focus on categorical persistence functions that allow us to study in full generality strong kinds of connectedness such as clique communities, $k$-vertex and $k$-edge connectedness directly on simple graphs and monic coherent categories.
% \keywords{}
%
% \subclass{68R10 \and 05C10 \and 18C99}
%
\end{abstract}


\section{Introduction}

Persistent homology allows for swift and robust comparison of topological objects. However, raw data are rarely endowed with a topological structure. Persistent homology and topological persistence are by their own nature bound to topological spaces. Thus, in applications, persistent homology is applied after data are arbitrarily mapped to topological spaces by means of auxiliary constructions (e.g.~\cite{zomorodian2005computing,kurlin2016fast,rieck2018clique}). Although these constructions have been employed successfully in several domains (e.g.~\cite{ferri2017persistent,expert2019topological}), they unavoidably alter the information carried by the original data set.

Rank-based persistence~\cite{bergomi2019rank} extends topological persistence to arbitrary categories from an axiomatic perspective and by means of category theory. In summary, the theory developed in~\cite{bergomi2019rank} allows to compute the persistence of objects in arbitrary source categories (rather than topological spaces) and consider any regular category as target category, whereas classical persistence is limited to vector spaces and sets.

Our main aim is to build on the aforementioned categorical generalization to allow for a more direct analysis of significant data types such as simple graphs and quivers, and guarantee the properties of stability and universality of the classical persistence framework. With this aim in mind, we introduce the definition of \textit{weakly directed properties} as a tool to easily build categorical persistence functions that describe graph-theoretical concepts of connectivity, e.g. clique communities, and $k$-vertex and $k$-edge connectedness. Thereafter, we extend these constructions to the more general setting of coherent categories.

In more details, in Section~\ref{sec:persistencesubobjects}, we provide the definition of monic categorical persistence function and show how the persistence diagrams associated to such functions can be described as multisets of points and half-lines, as in the classical framework. Furthermore, we define the natural pseudodistance in this general context and list the assumptions necessary to obtain tame filters.
We introduce in Section~\ref{sec:weaklydirectedproperties} our framework in the category of weakly directed posets. We define stable monic persistence functions in this category, and prove stability and universality. This construction allows us to define \textit{weakly directed properties} and describe the associated persistence functions. In particular, we show that the classical clique communities construction is a weakly directed property.
Generalizing further our framework to monic subcategories of coherent categories in Section~\ref{sec:generalizedconnectedness}, we describe $k$-vertex and $k$-edge connectedness. Thereafter, these constructions are exemplified and their stability and universality are discussed.

All constructions built through the proposed generalized persistence are discussed on the weighted graph depicted in Fig.~\ref{fig:toy_example}. For completeness, in the same figure we compute the persistence diagrams and persistent Betti numbers of the example graph seen as a simplicial complex.

\begin{figure}[tb]
    \centering
    \includegraphics[width = 0.9 \textwidth]{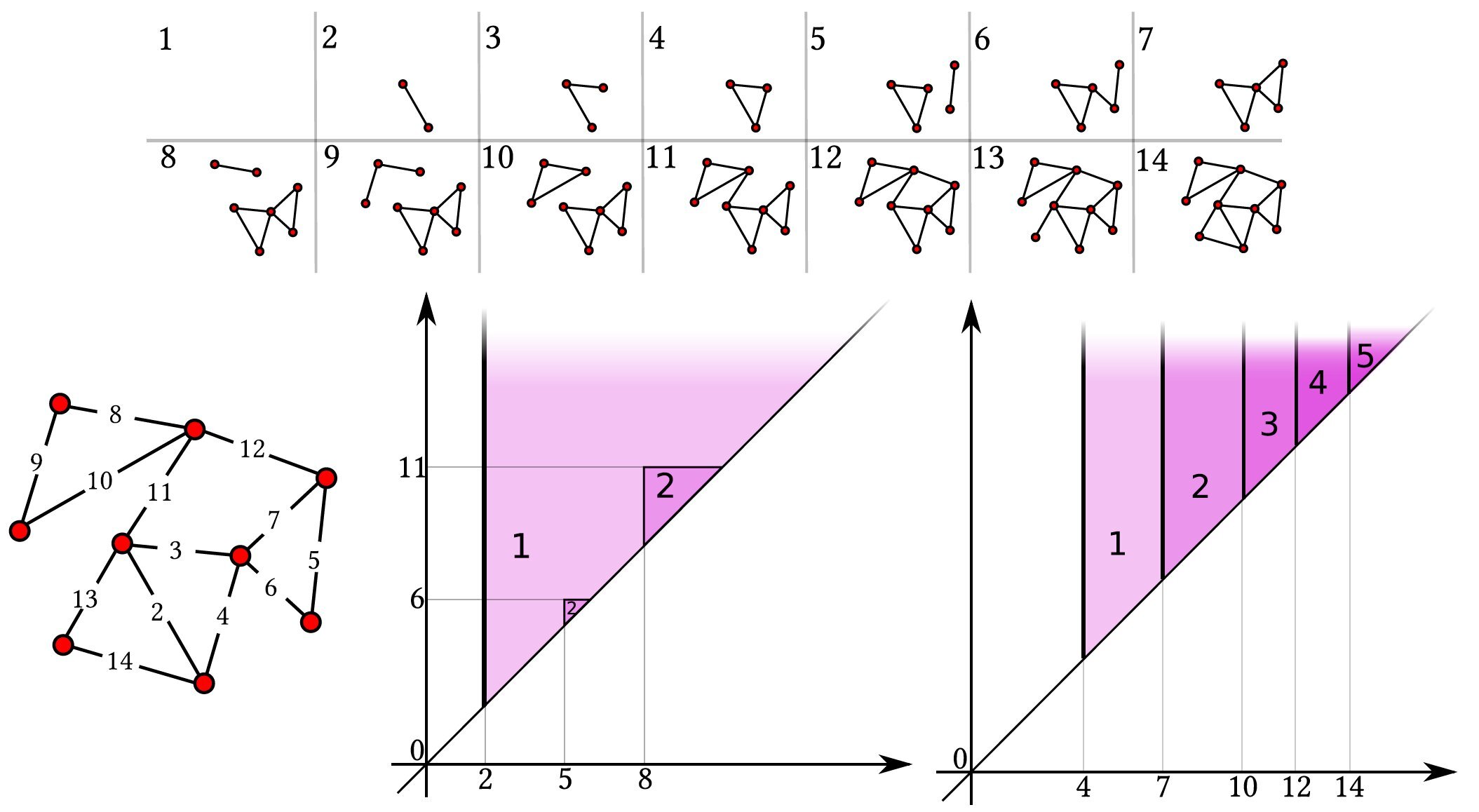}
    \caption{A weighted graph, and its persistence diagram and persistent Betti numbers.
    }\label{fig:toy_example}
\end{figure}

\section{Persistence via the poset of subobjects}
\label{sec:persistencesubobjects}

We give concrete applications of the framework developed in~\cite{bergomi2019rank}, which defines the notion of \emph{categorical persistence function}~\cite[Def. 3.2]{bergomi2019rank} in arbitrary categories. Unlike~\cite{bergomi2019rank}, here we restrict ourselves to filtrations, rather than arbitrary $\Rl$-indexed diagrams. In other words, given a category $\C$, we will consider categorical persistence functions in $\CM$, the subcategory of $\C$ where the only allowed morphisms are monomorphisms.

\begin{definition}\cite[Def. 3.2]{bergomi2019rank}
    \label{persfct}
    A \emph{persistence function} is a categorical persistence function on the category $\Rl$. It is a correspondence that maps each pair of real numbers $u \le v$, to an integer $p(u, v)$ such that, given $u_1 \le u_2 \le v_1 \le v_2$, the following inequalities hold.
    \begin{enumerate}
        \item $p(u_1, v_1) \le p(u_2, v_1)$ and $p(u_2, v_2) \le p(u_2, v_1)$, that is to say $p$ is non-decreasing in the first argument, and non-increasing in the second.
        \item $p(u_2, v_1) - p(u_1, v_1) \le p(u_2, v_2) - p(u_1, v_2)$.
    \end{enumerate}
\end{definition}

\begin{definition}\label{def:monic_persistent_function}
    Let $\C$ be an arbitrary category.
    A \emph{monic persistence function} on $\C$ is a categorical persistence function on $\CM$. It maps each inclusion $u \hookrightarrow v$ to an integer $p(u, v)$, such that, given $u_1 \hookrightarrow u_2 \hookrightarrow v_1 \hookrightarrow v_2$, the following inequalities hold.
    \begin{enumerate}
        \item $p(u_1, v_1) \le p(u_2, v_1)$ and $p(u_2, v_2) \le p(u_2, v_1)$.
        \item $p(u_2, v_1) - p(u_1, v_1) \le p(u_2, v_2) - p(u_1, v_2)$.
    \end{enumerate}
\end{definition}

A filtration $F$ in $\C$ can naturally be seen as a functor from $\Rl$ to $\CM$. Therefore, by functoriality, a monic persistence function on $\C$ and a filtration in $\C$ induce a persistence function. In turn, such persistence function $p$ induces a persistence diagram, which respects familiar properties. Importantly, the construction of persistence diagrams, as the familiar multiset of points and half-line segments is guaranteed by the following proposition.

\begin{proposition}\label{repr}
    For p and ${\cal D}F$ we have
    \[
        p_F{(\beta, \gamma)} =
        \sum_{\substack{(u, v) \in \Delta^*,\\
                u<\beta, \ v>\gamma}}\mu(u, v)
    \]
    for every $(\beta, \gamma) \in \Delta^*$ which is no discontinuity point of $p_F$.
\end{proposition}
\begin{proof}
    By applying~\cite[Prop. 3.17]{bergomi2019rank} with $\alpha<\bar{a}$ and $\delta = +\infty$.
\end{proof}

This result implies that the discontinuity sets of $p_F$ are either vertical or horizontal (possibly unbounded) segments with end-points in the cornerpoints. This means that persistence functions have the appearance of superimposed triangles, typical of Persistent Betti Number functions.

\paragraph{Natural pseudodistance.} It is possible to define the natural pseudodistance on $\CM$. In the following, we will adopt the following finiteness assumptions, to ensure stability of all persistence functions we will consider. Note that, whenever we refer to categories such as $\S, \P, \G$, we always refer to the finite version---finite sets, finite posets, and finite simple graphs.

\paragraph{Finiteness assumptions.} From now on, we assume that every object in $\C$ has only a finite number of distinct subobjects (to ensure tameness in all constructions). Furthermore, we will only consider filtrations $F$ that admit a colimit $F(\infty)$ in $\C_m$. As every object has only a finite number of distinct subobjects, this means that $F(x \le x')$ is an isomorphism for sufficiently large $x, x'$. This will allow us to define the {\em natural pseudodistance}~\cite{FrMu99,dAFrLa10,Les15}.

\begin{definition}\label{natural}
    Let $F_1, F_2$ be two filtrations in $\C$. Let $\mathcal{H}$ be the (possibly empty) set of isomorphisms between $F_1(\infty)$ and $F_2(\infty)$. Given an isomorphism $\mathcal{H}\ni\phi\colon F_1(\infty) \rightarrow F_2(\infty)$, we can consider the set
    \begin{equation*}
        L_\phi = \{h\in\R_{\ge 0} | \text{for all } x\in\R, \; F_2(x-h) \subseteq \phi(F_1(x)) \subseteq F_2(x+h)\},
    \end{equation*}
    where the inclusion is among subobjects of $F_2(\infty)$.
    The {\em natural pseudodistance} between $F_1$ and $F_2$ is
    \begin{equation*}
        \delta\left(F_1, F_2\right) =
        \inf \bigcup_{\phi\in\mathcal{H}} L_\phi.
    \end{equation*}
\end{definition}
The natural pseudodistance is equal to the interleaving distance, when considering $F_1, F_2$ as $\Rl$-indexed diagrams in $\CM$. By the universal property of colimits, a strong interleaving induces an isomorphism $F_1(\infty) \simeq F_2(\infty)$, which behaves correctly on sublevels. Conversely, applying $\phi$ and its inverse on sublevels induces a strong interleaving.

\subsection{Preliminaries on posets}

The first category we consider is $\WP$, the category of finite {\em weakly directed} posets.

\begin{definition}
    \label{def:weaklydirectedposet}
    A poset $P$ is {\em weakly directed} if, whenever $a,b\in P$ have a lower bound, they also have an upper bound.
\end{definition}

Stong~\cite{stong1966finite} discusses the analogy between finite posets and finite $T0$ topological spaces. The two categories are equivalent, and it is therefore possible to consider the homotopy type of a finite poset. In particular, Stong shows a procedure to determine whether two posets have the same homotopy type. For the sake of self-containment, we report here a description of the procedure from~\cite{raptisHomotopyTheoryPosets2010}.

Let $P$ be a poset. An element $p\in P$ is {\em upbeat} (resp. {\em downbeat}) if the set of all element strictly (resp. lower) than $p$ has a minimum (resp. maximum). The insertion of deletion of an upbeat or downbeat point does not change the strong homotopy type of $P$. The core of $P$, denoted $\core(P)$, is a deformation retract of $P$ that is minimal; i.e. it does not contain neither upbeat nor downbeat elements. One can always reach $\core(P)$ by successively deleting beat points from $P$.

\begin{theorem}\cite[Thm.~4]{stong1966finite}
    Two finite posets are strongly homotopy equivalent if and only if they have isomorphic cores.
\end{theorem}

We are now ready to show that there is a canonical homotopy equivalence between weakly directed finite posets and finite sets.

\begin{proposition}
    Let $\free\colon \S \rightarrow \WP$ be the free poset functor, i.e. the functor that associates to a finite set $S$ the weakly directed poset $(S, =)$. $\free$ admits a left adjoint $\maximal$. Furthermore, let
    \begin{equation*}
     \epsilon\colon \maximal\circ \free \rightarrow \id_{\S}
     \qquad\text{ and }\qquad
     \eta\colon \id_{\WP} \rightarrow \free\circ \maximal
    \end{equation*}
    be the natural transformations associated to the adjunction. $\epsilon$ is a natural isomorphism, whereas $\eta$ is a natural homotopy equivalence.
\end{proposition}
\begin{proof}
    $\maximal$ associates each weakly directed poset to the set of its maximal elements. This mapping can be extended to a functor, as for each order-preserving map $f\colon P\rightarrow P'$, given a maximal element $l\in P$, there is a unique maximal element $l'\in P'$ with $f(l)\le l'$.
    Given a set $S$, the maximal elements of the poset $(S, =)$ are clearly all elements of $S$, so $\epsilon$ is a natural isomorphism.
    For a weak-directed poset $P$, the map $\eta_P \colon P \rightarrow \free(\maximal(P))$ is a deformation retract of $P$ onto its core. To see this, starting from $P$, we can proceed by removing elements that are maximal in $P\setminus \maximal(P)$. If an elements is maximal in $P\setminus \maximal(P)$, it is necessarily upbeat: distinct maximal elements in $P$ can have no lower bound. After iteratively removing all elements in $P\setminus \maximal(P)$, we obtain the desired deformation retract $\eta_P \colon P \rightarrow \free(\maximal(P))$.
\end{proof}

The functor $\maximal\colon \WPM \rightarrow \S$ induces a monic persistence function on $\WP$, by \cite[Prop.~3.6]{bergomi2019rank}. Furthermore, such persistence function factors via a ranked category with finite colimits, $\S$, and is therefore stable by~\cite[Thm.~3.27]{bergomi2019rank}.

Universality is generally not granted for stable persistence functions.
We now follow the logical line of Thm.~32 of \cite{dAFrLa10} for proving the universality of the bottleneck (or matching) distance among the lower bounds for the natural pseudodistance that can come from distances between persistent block diagrams.

Let $F$ be a filtration in $\WP$. If $F(\infty)$ has several maximal elements, then all maximal elements arising in the filtration are bounded by one of them, so the construction can be performed for each of the lower set of maximal points of $F(\infty)$.

\begin{proposition}\label{prop:posetuniversal}
    Let $F, F'$ be to filtrations in $\WP$; let $D(F)$, \ $D(F')$ be the respective persistence diagrams. Then there exist filtrations $H, H'$ such that
    \begin{enumerate}
        \item $D(F)=D(H)$, \ $D(F')=D(H')$,
        \item $d\big(D(H), D(H')\big) = \delta\big(H, H'\big)$,
    \end{enumerate}
    where $d$ is the bottleneck distance between persistence diagrams.
\end{proposition}
\begin{proof}
    There is at least one bijection $\gamma$ between the multisets $D(F)$ and $D(F')$ which realizes the distance $d = d\big(D(F), D(F')\big)$. There are $p_0, p_0'$, cornerpoints at infinity of $D(F)$, \ $D(F')$ respectively, and proper cornerpoints $p_1, \ldots, p_m$, $p'_1, \ldots, p'_m$. Note that here we consider also cornerpoints on the diagonal, but only those that are needed to realize the matching distance. Up to relabeling the cornerpoints, we can assume that the matching is given by $p_i \mapsto p_i'$ for $i \in \{0, \dots, m\}$. We denote $x_i, y_i$ (resp. $x_i', y_i'$) the coordinates of the cornerpoint $p_i$ (resp. $p_i'$).
    The distance $d$ is then the maximum of the distance in the $L^\infty$ norm of corresponding points. We now construct new filtrations of posets $H, H'$ as follows.
    \begin{equation*}
        \begin{aligned}
            &H(x) = \{p_i | x_i \le x\}\\
            &p_i < p_j \text{ if } j=0 \text{ and } y_i \le x,
        \end{aligned}
        \qquad\text{ and }\qquad
        \begin{aligned}
            &H'(x) = \{p_i' | x_i' \le x\}\\
            &p_i' < p_j' \text{ if } j=0 \text{ and } y_i' \le x.
        \end{aligned}
    \end{equation*}
    Choosing the isomorphism $H(\infty) \rightarrow H'(\infty)$ given by $p_i \mapsto p_i'$, we can show that the pseudodistance between $H$ and $H'$ is smaller or equal than $d$. By stability, it must be equal.
\end{proof}

\subsection{Weakly directed properties}\label{sec:weaklydirectedproperties}

% We will discuss a general strategy to prove it for each example discussed in the following sections. Our strategy is simple. We first show it for persistence induced by $\pi_0$ on $\G$. Then, we show that the persistence functions we consider are as rich as the connected components, via a factorization trick.
Some of the most informative graph-theoretical concepts describe the local connectivity of a graph from different viewpoints, e.g. considering the number of edges to be removed to disconnect it. The following definitions express these kind of stronger connectivities that we will use as categorical persistence functions, e.g. connected components, path-connected components (giving rise to 0-Betti numbers in \v{C}ech and singular homology respectively).

\begin{definition}
    \label{def:weaklydirectedproperty}
    Let $\C$ be a category and let ${\cal P}\subseteq \Obj(\C) / \simeq$ be a \emph{property} that is preserved by isomorphisms. We call $S_{\cal P}$ the functor $\CM \rightarrow \PM$ that associates to each object in $\C$ the poset of its subobjects that respect the property $\cal P$. We say that the property $\cal P$ is {\em weakly directed} if, for all $X \in\Obj(\C)$, $S_{\cal P}(X)$ is a weakly directed poset.
\end{definition}

% We recall that posets have a notion of connected components, which are the equivalence classes of the equivalence relation generated by the order~\cite{schroder2003ordered}.
% Let us denote $\pi_0\colon \mathbf{Poset}\rightarrow\S$ the connected components functor.

\begin{proposition} \label{prop:persistencefromproperty}
    Let $\C$ be a category, and let $\CM$ be the subcategory of $\C$ where the only allowed morphisms are monomorphisms.
    Let ${\cal P}$ be a weakly directed property on $\Obj(\C)$. Then $\cal P$ induces a stable categorical persistence function on $\CM$ which we denote $p_{\cal P}$.
\end{proposition}
\begin{proof}
    We can consider the functor $S_{\cal P}\colon \CM\rightarrow\WPM$. As $\WPM$ is equipped with a persistence function, this induces a persistence function on $\CM$ by~\cite[Prop. 3.3]{bergomi2019rank}.
\end{proof}

Unlike stability, universality of $p_{\cal P}$ is in general not guaranteed. However, the following condition is sufficient to ensure it.

\begin{proposition}
    \label{prop:universalityproperty}
    Let ${\cal P}, p_{\cal P}, S_{\cal P}$ be as in Def.~\ref{def:weaklydirectedproperty} and Prop.~\ref{prop:persistencefromproperty}. Let us further assume that there exists a functor $T\colon \WPM \rightarrow \CM$ such that $\maximal\circ S_{\cal P}\circ T$ is naturally isomorphic to $\maximal$. That is to say, for all $P\in\Obj(\WPM)$, the maximal elements of $S_{\cal P}(T(P))$ are in a one-to-one correspondence with the maximal elements of $P$, and this bijection is natural in $P$. Then the bottleneck distance between diagrams induced by $p_{\cal P}$ is universal with respect to the natural pseudodistance.
\end{proposition}
\begin{proof}
    Given two filtrations $F, F'$ in $\C$, we can consider the filtrations of posets $S_{\cal P}\circ F$ and $S_{\cal P}\circ F'$. By Prop.~\ref{prop:posetuniversal}, there are filtrations of weakly directed poset $H, H'$ with the same persistence diagram, whose interleaving distance equals the bottleneck distance. Then, $T\circ H$ and $T\circ H'$ have the same persistence diagram as $F, F'$, and their interleaving distance equals the bottleneck distance.
\end{proof}

\subsubsection*{Clique communities}\label{communities}

An example of weakly directed property comes from clique communities.
We recall the definition of clique community given in \cite{Pa*05}. Given a graph $G=(V, E)$, two of its $k$-cliques (i.e. cliques of $k$ vertices) are said to be {\em adjacent} if they share $k-1$ vertices; a $k-${\em clique community} is a maximal union of $k$-cliques such that any two of them are connected by a sequence of $k$-cliques, where each $k$-clique of the sequence is adjacent to the following one. This construction has been applied to  network analysis \cite{toivonen2006model,kumpula2007emergence,palla2007quantifying,fortunato2010community} and to weighted graphs, in the classical topological persistence paradigm, in~\cite{rieck2018clique}. Here we consider a weighted graph as a filtration of graphs (where the weight of each vertex is the $\inf$ of the weights of its incident edges).

\begin{definition}
    \label{def:cliqueproperty}
    A graph $G$ belongs to $c^k$ if it is union of $k$-cliques, such that any two of them are connected by a sequence of adjacent $k$-cliques.
\end{definition}

\begin{figure}[tb]
    \centering
    \includegraphics[width = 0.75 \textwidth]{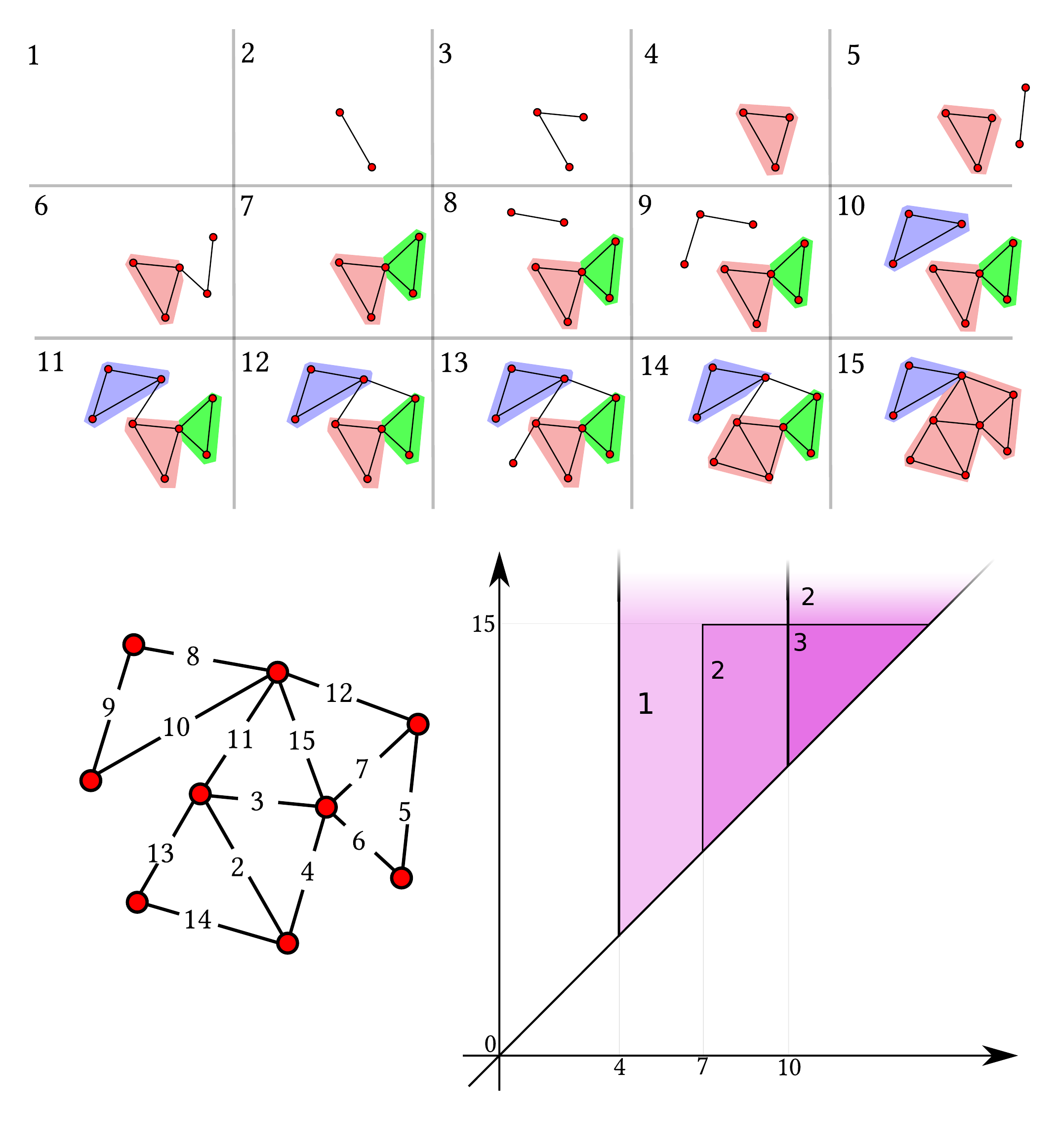}
    \caption{Example of persistent 3-clique community number function.
    }\label{figcommunity}
\end{figure}

\begin{proposition}
    $c^k$ is a weakly directed property.
\end{proposition}
\begin{proof}
    If two subgraphs $G_1, G_2 \subseteq G$ are in $c^k$, and there is a $k$-clique in $G_1 \cap G_2$, then $G_1\cup G_2$ is also in $c^k$.
\end{proof}
As a consequence, $c^k$ induces a stable persistence function $p_{c^k}$ on graph filtrations, which we call {\em persistent $k$-clique community number}. In practice, given a graph filtration $F$, the persistent $k$-clique community number $p_{c^k}(u, v)$ equals the number of $k$-clique communities in $F(v)$ that contain at least a $k$-clique when restricted to $F(u)$.

\begin{remark}
    Of course, the persistent 2-clique community number function of a weighted graph $(G, f)$, such that no isolated vertices appear in the filtration, coincides with its persistent 0-Betti number function.
\end{remark}

An example of persistent 3-clique community number function can be seen in Fig.~\ref{figcommunity}. We can associate to $p_{c^k}$, via \cite[Def.~3.13]{bergomi2019rank}, a {\em persistent $k$-clique community diagram} $D_{c^k}(f)$.

There is a general construction that allows us to prove universality for many weakly directed properties of graphs. Let $n$ be a positive integer. Given a weakly directed poset $P$, we can consider the graph whose vertices are $P\times \{1, \dots, n\}$, where distinct vertices $(p, i)$ and $(q, j)$ are connected by an edge if $p$ and $q$ are comparable in $P$. This mapping induces a functor $T_n\colon \WPM \rightarrow \GM$, and $\maximal\circ S_{c^k}\circ T_k$ is naturally isomorphic to $T_k$. By Prop.~\ref{prop:universalityproperty}, the bottleneck distance is universal with respect to the natural pseudodistance.

\section{Generalized connectedness}
\label{sec:generalizedconnectedness}

Here, we introduce a notions of generalized connectedness in the context of \emph{coherent categories}, i.e. categories with pullback-stable image factorization and pullback-stable union of subobjects.
For nomenclature and basic facts about coherent categories we follow~\cite[Chapt. A1.4]{johnstoneSketchesElephantTopos2002}. Other authors refer to coherent categories with different names, such as \emph{pre-logos}~\cite{freydCategoriesAllegories1990}, or \emph{logical category}~\cite{makkaiFirstOrderCategorical2006}. The reader can think as an example $\C = \G$, the category of finite simple graphs, but many variations are possible, such as quivers, simplicial sets, simplicial complexes, or, more generally, categories of sheaves or concrete sheaves.

First, we introduce the notion of connected object in a coherent category.

\begin{definition} \label{def:connected}
    An object $C$ in a coherent category $\mathbf C$ with initial object $0$ is \emph{connected} if, for all pairs of subobjects $A, B\hookrightarrow C$ with $A\cap B = 0$ and $A\cup B = C$, exactly one of $A$ and $B$ is $0$.
\end{definition}
\begin{remark}
    In the above definition, we consider equality as equality of subobjects. So $A \cap B = 0$ means that $A\cap B$ is isomorphic to $0$, and $A \cup B = C$ means that the the monomorphism $A\cup B\hookrightarrow C$ induced by $A\hookrightarrow C$ and $B\hookrightarrow C$ is an isomorphism.
\end{remark}
The notion of connectedness is generally introduced, in a categorical setting, in the context of {\em extensive categories}~\cite{CaLaWa93}. There, we say that a non-initial object $C$ is connected if it cannot be expressed as a non trivial coproduct $A\amalg B$.
Here, however, the related but distinct concept of coherent category is needed, as we will need to work with the distributive lattice of subobjects. If a category $\C$ is both coherent and extensive, however, the two notions overlap. As coproducts are disjoint in an extensive category, $A, B$ are always disjoint subobjects of $A\amalg B$. Conversely, in a coherent category, if $A\cap B = 0$, then $A\cup B \simeq A\amalg B$, with coprojections given by the inclusions.

\begin{proposition}\label{prop:connectedweaklydirected}
    Let $X_0, X_1, X_2$ be subobjects of $X\in\Obj(\C)$. If $X_0 \subseteq X_1$ and $X_0 \subseteq X_2$ and all $X_0, X_1, X_2$ are connected, so is $X_1\cup X_2$.
    As a consequence, the property of being connected is weakly directed.
\end{proposition}
\begin{proof}
    Let us consider the following diagram of subobjects of some object in $\C$:
    \begin{equation*}
        \begin{tikzcd}
            X_0 \arrow[hookrightarrow]{r} \arrow[hookrightarrow, swap]{d}
            & X_1 \arrow[hookrightarrow]{d}\\
            X_2 \arrow[hookrightarrow]{r} & X1 \cup X2 \\
        \end{tikzcd}
    \end{equation*}
    By hypothesis $X_0, X_1, X_2$ are connected.
    Let $A, B$ be subobjects of $X_1 \cup X_2$ such that $A\cup B = X_1 \cup X_2$ and $A\cap B = 0$. $X_0$ is connected and $X_0 = (A\cap X_0)\cup(B\cap X_0)$ by distributivity~\cite[Lm. 1.4.2]{johnstoneSketchesElephantTopos2002}.
    Therefore, exactly one of $A\cap X_0$ and $B\cap X_0$ is initial. Let us assume without loss of generality $A\cap X_0 = 0$ and $B \cap X_0 = X_0$. Then $A\cap X_1 = 0$ and $A\cap X_2 = 0$, whereas $B\cap X_1 = X_1$ and $B\cap X_2 = X_2$. As a consequence, $A = 0$ and $B = X_1 \cup X_2$. From this, we conclude that $X$ is connected and $X_1 \cup X_2$ is $\cal F$-connected.
\end{proof}

In the coherent category $\G$ the notion of connectedness can be generalized to stronger notions, namely $k$-vertex-connectedness~\cite{balinskiGraphStructureConvex1961} and $k$-edge-connectedness [ref?].
The same holds in an arbitrary coherent category $\C$, provided one is able to specify a class of monomorphism that is {\em preserved by restrictions}, such as deleting fewer than $k$-vertices or fewer than $k$-edges in a graph.
Here, we provide a unified notion of $\cal F$-connected component for a given class $\cal F$ of monomorphisms in $\mathbf{C}$.
We prove that being ${\cal F}$-connected is a weakly directed property in the sense of Def.~\ref{def:weaklydirectedproperty}.

\begin{definition}
    \label{def:preservedbyrestrictions}
    Let $\mathbf{C}$ be a coherent category. We say that a family $\cal{F}$ of monomorphisms in $\CM$ is {\em preserved by restrictions} if, for each monomorphism $\phi \in \cal{F}$, pullbacks of $\phi$ along monomorphisms are also in $\cal{F}$.
\end{definition}
\begin{remark}
    As pullbacks are only defined up to isomorphism, we are tacitly assuming that if $\phi$ is in ${\cal F}$, and $\psi, \psi'$ are isomorphisms, then $\psi \circ \phi \circ \psi'$ is also in $\cal F$.
\end{remark}

\begin{definition} \label{def:fconnected}
    Let $\C$ be a coherent category. Let $\cal{F}$ be a family of monomorphisms in $\CM$ that is preserved by restrictions. An object $X$ is $\cal{F}$-{\em connected} if all subobjects $Y \xhookrightarrow{\phi} X$, where the inclusion $\phi$ belongs to $\cal{F}$, are connected. Given an object $C \in \textnormal{Obj}(\C)$, we say that an $\cal{F}$-{\em connected component} of $C$ is a maximal subobject $X \hookrightarrow C$ such that $X$ is $\cal{F}$-connected.
\end{definition}

\begin{proposition}\label{prop:fconnetedweaklydirected}
    $\cal F$-connectedness is a weakly directed property, in the sense of Def.~\ref{def:weaklydirectedproperty}.
\end{proposition}
\begin{proof}
    Let us consider the following commutative diagram of subobjects $X_0, X_1, X_2$ of some object $C \in \textnormal{Obj}({\C})$.
    \begin{equation}\label{diag:subobjects}
        \begin{tikzcd}
            X_0 \arrow[hookrightarrow]{r} \arrow[hookrightarrow, swap]{d}
            & X_1 \arrow[hookrightarrow]{d}\\
            X_2 \arrow[hookrightarrow]{r} & X_1 \cup X_2 \\
        \end{tikzcd}
    \end{equation}
    We need to prove that if $X_0, X_1, X_2$ are $\cal{F}$-connected, so is $X_1\cup X_2$.
    Let $Y \xhookrightarrow{\phi} X_1\cup X_2$ be a monomorphism in $\cal{F}$.
    Let us consider the pullback of diagram~\ref{diag:subobjects} along $\phi$:
    \begin{equation*}
        \begin{tikzcd}
            Y_0 \arrow[hookrightarrow]{r} \arrow[hookrightarrow, swap]{d}
            & Y_1 \arrow[hookrightarrow]{d}\\
            Y_2 \arrow[hookrightarrow]{r} & Y \\
        \end{tikzcd}
    \end{equation*}
    $Y = Y_1 \cup Y_2$, as finite unions of subobjects are preserved by pullbacks in coherent categories.
    By hypothesis $Y_0, Y_1, Y_2$ are connected and, by Prop~\ref{prop:connectedweaklydirected}, so is $Y$.
\end{proof}

In the rest of this section, we will give several examples of ${\cal F}$-connectedness, and study the associated persistence functions.

\subsection{Blocks}\label{blocks}

% We recall that in a (loopless) graph $G$ a \textit{cut vertex} (or {\it separating vertex}) is a vertex $v\in V(G)$ whose deletion (along with incident edges) makes the number of connected components of $G$ increase.
% A graph $G$ is \textit{$k$-connected} if it is connected, has at least two vertices, and does not contain any cut vertex~\cite{BoMu11}.
We recall that a graph is {\em $k$-vertex-connected} if it has at least $k$-vertices and remains connected whenever fewer than $k$ vertices are removed~\cite{balinskiGraphStructureConvex1961,harary1962maximum}. We say that a maximal $k$-vertex-connected subgraph of a given graph $G$ is a {\em $k$-vertex-connected component}.

\begin{definition}
    \label{def:kvertexdeletion}
    Let $G_1 = (V_1, E_1)$ and $G_2 = (V_2, E_2)$ be two finite simple graphs, with $G_1\subseteq G_2$.
    We say that the inclusion $G_1\subseteq G_2$ is a {\em $k$-vertex deletion} if:
    \begin{itemize}
        \item $G_1$ is a full subgraph of $G_2$,
        \item $\lvert V_2 \rvert - \lvert V_1 \rvert < k$.
    \end{itemize}
\end{definition}
The class of $k$-vertex deletions is preserved by restrictions in the sense of Def.~\ref{def:preservedbyrestrictions}: a $k$-vertex-deletion restricted to a subobject is still a $k$-vertex-deletion. Therefore, this class of monomorphisms induces a notion of generalized connectedness.

Let us denote $v^k$ be the class of graphs that are $\cal F$-connected, with $\cal F$ the class of $k$-vertex deletions. $v^k$ is a weakly directed property, and therefore it induces a functor
\begin{equation*}
    S_{v^k}\colon\GM\rightarrow\WPM
\end{equation*}
and a stable monic persistence function $p_{v^k}$ on graph filtrations, which we call {\em persistent $k$-block number}.
In practice, given a graph filtration $F$, the persistent $k$-block number $p_{c^k}(u, v)$ equals the number of $k$-vertex-connected components in $F(v)$ that contain at least a $k$-vertex-connected component when restricted to $F(u)$.

Furthermore, the bottleneck distance is universal with respect to the natural pseudodistance. To prove universality, we consider the same functor $T_n$ as in Sect.~\ref{communities}, and note that $\maximal\circ S_{v^k}\circ T_k$ is naturally isomorphic to $\maximal$.

An example of $k$-block number function (for $k=2$) can be seen in Fig.~\ref{figblock}. We can then associate to $p_{v^k}$, via \cite[Def.~3.13]{bergomi2019rank}, a {\em persistent block diagram} $D_{v^k}(f)$ with all classical features granted by the propositions of Section~\ref{sec:persistencesubobjects}.
A toy example is given in Fig.~\ref{figblock}.

\begin{figure}[tb]
    \centering
    \includegraphics[width = 0.75\textwidth]{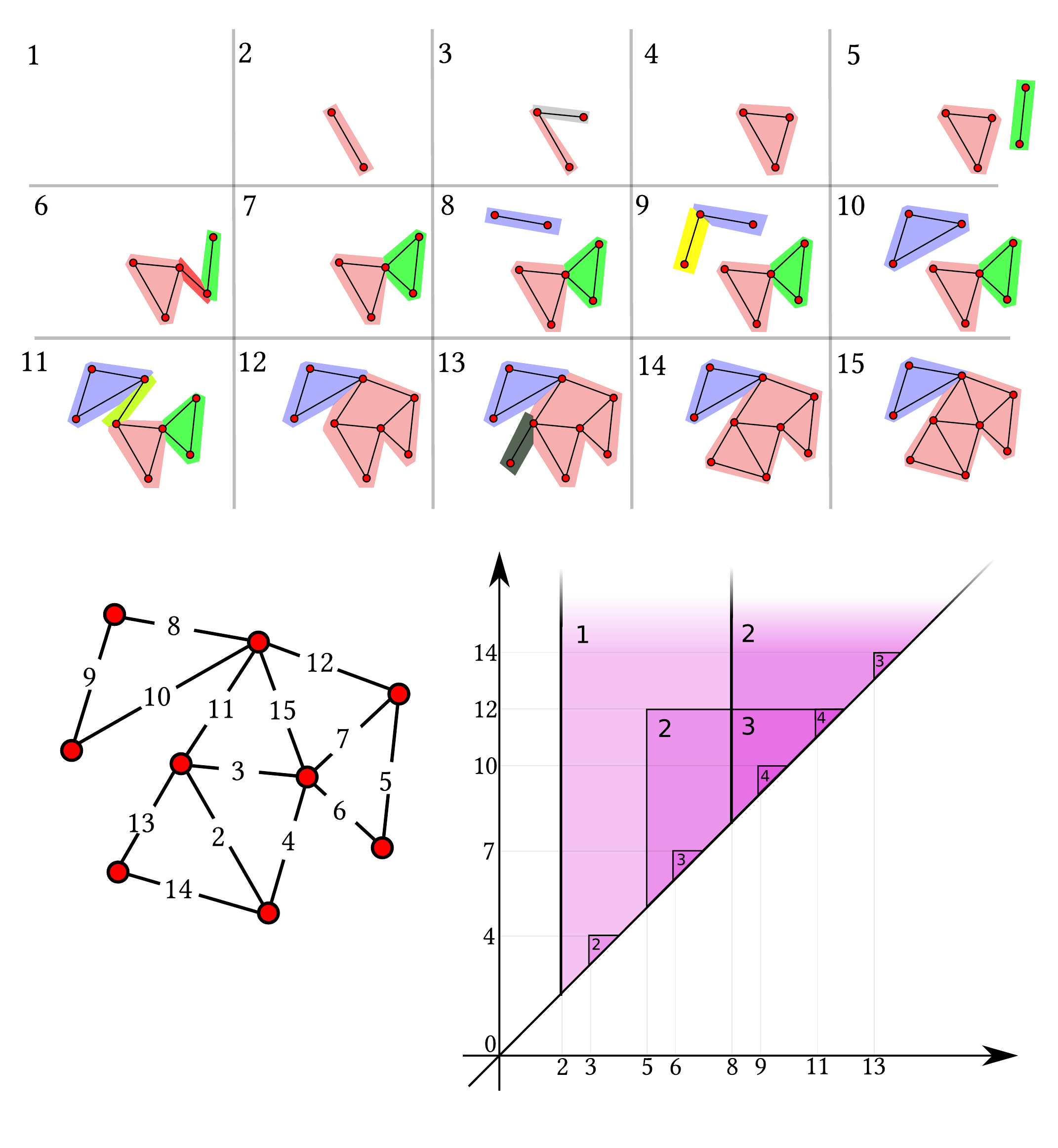}
    \caption{Example of persistent block number function.
    }\label{figblock}
\end{figure}

\subsection{Edge-blocks}\label{eb}
% We recall that in a graph $G$ a \textit{cut edge} (or {\it bridge}) is an edge $e\in E(G)$ whose deletion  makes the number of connected components of $G$ increase \cite{BoMu11}.
We say that a graph is {\em $k$-edge-connected} if it remains connected when removing fewer than $k$ edges~\cite{harary1962maximum}.

\begin{definition}
    \label{def:kedgedeletion}
    Let $G_1 = (V_1, E_1)$ and $G_2 = (V_2, E_2)$ be two finite simple graphs, with $G_1\subseteq G_2$.
    We say that the inclusion $G_1\subseteq G_2$ is a {\em $k$-edge deletion} if $\lvert E_2 \rvert - \lvert E_1 \rvert < k$.
\end{definition}
The class of $k$-edge deletions is preserved by restrictions in the sense of Def.~\ref{def:preservedbyrestrictions}: a $k$-edge-deletion restricted to a subobject is still a $k$-edge-deletion. Therefore, this class of monomorphisms induces a notion of generalized connectedness.
Let us denote $e^k$ be the class of graphs that are $\cal F$-connected, with $\cal F$ the class of $k$-edge deletions. $e^k$ is a weakly directed property, and therefore it induces a functor
\begin{equation*}
    S_{e^k}\colon\GM\rightarrow\WPM
\end{equation*}
and a stable monic persistence function $p_{e^k}$ on graph filtrations, which we call {\em persistent $k$-edge-block number}.
In practice, given a graph filtration $F$, the persistent $k$-edge-block number $p_{c^k}(u, v)$ equals the number of $k$-edge-connected components in $F(v)$ that contain at least a $k$-edge-connected component when restricted to $F(u)$.

Furthermore, the bottleneck distance is universal with respect to the natural pseudodistance. To prove universality, we consider the same functor $T_n$ as in Sect.~\ref{communities}, and note that $\maximal\circ S_{e^k}\circ T_k$ is naturally isomorphic to $\maximal$.

\begin{figure}[tb]
    \centering
    \includegraphics[width = 0.75\textwidth]{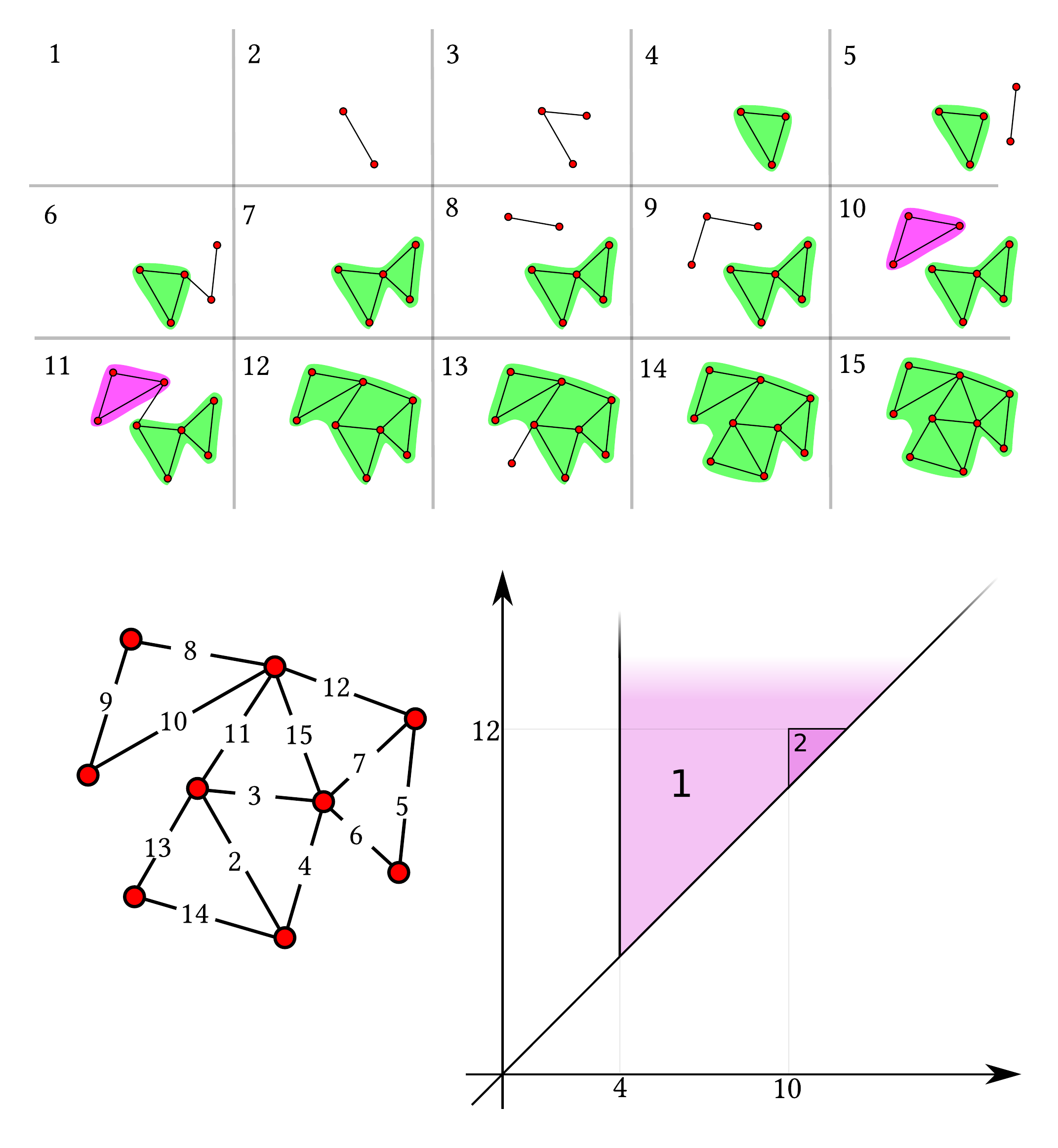}
    \caption{Example of persistent edge-block number function. For displaying purposes edge-blocks corresponding to isolated vertices are not shown.
    }\label{figeblock}
\end{figure}

An example of persistent edge-block number function (for $k=2$) can be seen in Fig.~\ref{figeblock}. We can associate to $p_{e^k}$, via \cite[Def.~3.13]{bergomi2019rank}, a {\em persistent edge-block diagram} $D_{e^k}(f)$.

\subsection{Blocks in group actions on quivers}

\begin{figure}[!htb]
    \centering
    \includegraphics[width = \textwidth]{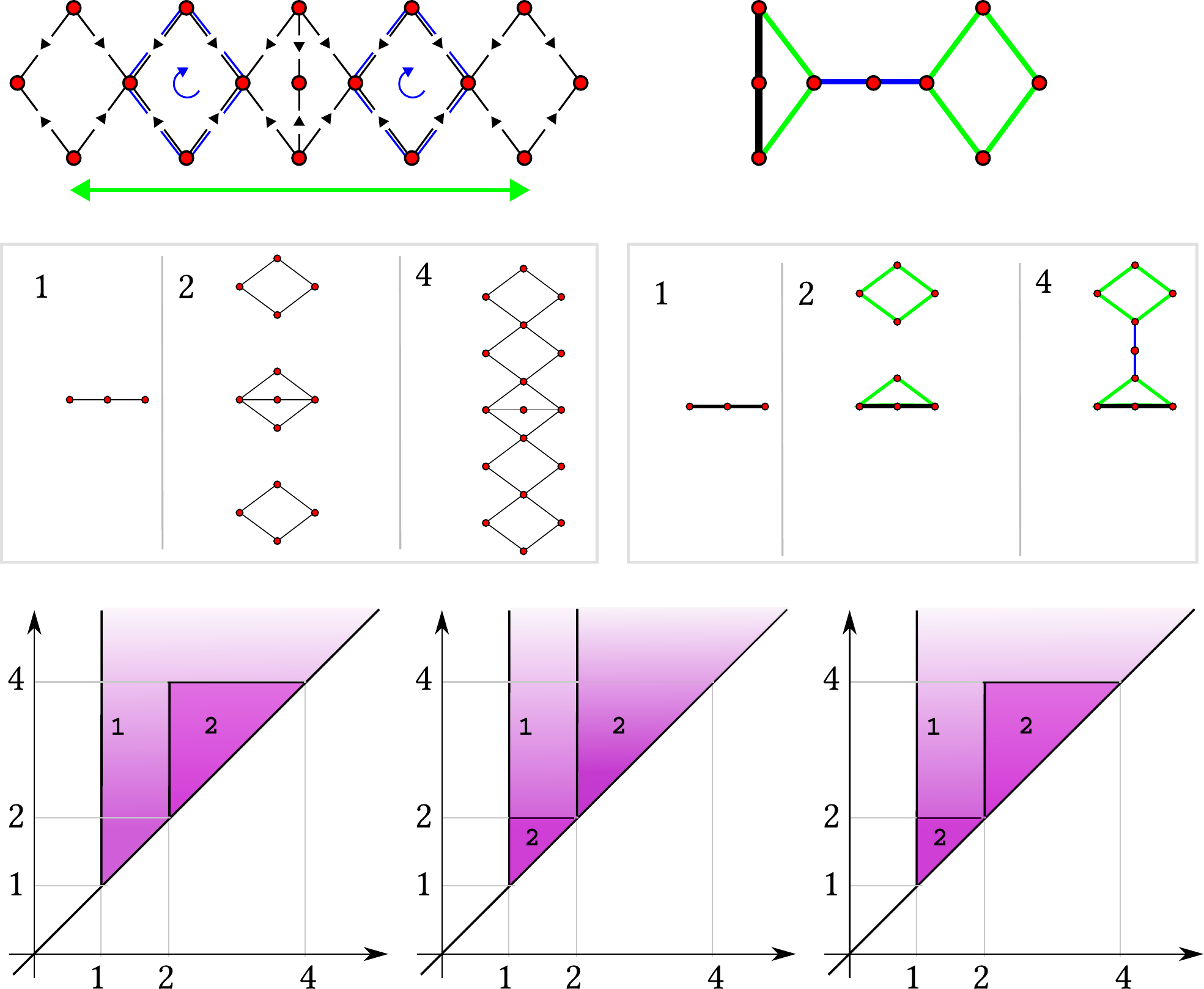}
    \caption{Persistence of a quiver $K$ with a group action of $\mathbb{Z}_2 \times \mathbb{Z}_2$. In the top row, on the right a graphical representation of $K$ and the group action: the green arrow is a reflection with respect to the center of the graphical representation of $K$ along the vertical axis. The second copy of $\mathbb{Z}_2$ acts on the edges highlighted in blue by identifying opposite edges with coherent orientations. On the right the quotient $K/\left(\mathbb{Z}_2 \times \mathbb{Z}_2\right)$. Colors correspond to the action generating the quotient; fixed points are depicted in black. In the middle row, the filtration induced by the cardinality of the orbit and its respective quotients. In the bottom row, the persistence diagrams obtained by considering connected components with respect to isomorphisms, orbit deletions, or fixed vertex deletions in $G-\mathbf{Quiver}$.}\label{fig:groupaction}
\end{figure}

Even though all previous examples were in the category of graphs, this frameworks can be used to effortlessly construct novel examples of persistence general coherent categories. Here we will consider $\mathbf{Quiver}$ the category of finite quivers (directed multigraphs, allowing self loops).

If a category $\mathbf{C}$ is coherent, then so is $\textnormal{Fun}(\mathbf{D}, \mathbf{C})$, where $\mathbf{D}$ is an arbitrary small category~\cite[Chapt. A1.4]{johnstoneSketchesElephantTopos2002}. In particular, given a group $G$, we can consider $G-\mathbf{Quiver}$, the category of $G$-actions on quivers. It is a functor category $\textnormal{Fun}(\mathbf{G}, \mathbf{Quiver})$, where $\mathbf{G}$ is a category with only one object $*$, and such that $\textnormal{Hom}(*, *) = G$, whose composition of morphisms is given by the operation in $G$.

Here we can consider at least three distinct classes of monomorphisms that are preserved by subobjects: isomorphisms, orbit deletions, or fixed vertex deletions. In turn, these three notions induce three distinct stable persistence functions, whose diagrams are illustrated in an example case in Fig.~\ref{fig:groupaction}.

\section{Conclusion and perspectives}

%We described a framework for data analysis framework via categorical persistence that can be easily adapted to diverse data types and representations while guaranteeing robustness and stability.

We built on a generalized theory of persistence, which no longer requires topological mediations such as auxiliary simplicial constructions, or the usage of homology as the functor of choice. We identified the properties that make a category suitable for our axiomatic persistence framework. We showed how these hypotheses allow us to define persistence directly in many relevant categories (e.g. graphs and simplices) and functor categories (simplicial sets and quivers), while guaranteeing the basic properties of classical persistence.
% We defined the generalized persistence functions and discussed their link with the natural pseudodistance.
We gave a flexible definition---weakly directed property---for the construction of generalized persistence functions, and applied them to toy examples in the category of weighted graphs. Therein, we discussed the stability of the generalized persistence functions built according to our definitions and by considering blocks, edge-blocks, and clique communities. Finally, as a confirmation of both generality and agility of our framework, we showed how various concepts of connectivity specific to graphs, such as blocks and edge-blocks are easily extended to other categories, in particular categories of presheaves, where they naturally induce (generalized) persistence.

This work is the combinatorial counterpart of the foundational results exposed in~\cite{bergomi2019rank}.
There, we focused on Abelian categories and categories of representations, aiming to extend persistence to objects relevant in theoretical physics or theoretical chemistry: Lie-group representations, quiver representations, or representations of the category of cobordisms (related to topological quantum field theory in \cite{baez_higherdimensional_1995}).
Here, we focus on coherent categories, which generally do not have any additive structure, but are of interest in several branches of Machine Learning and Artificial Intelligence. Graphs and their connectivity properties occupy a central role in these research fields, where networks are oftentimes represented as weighted quivers. These structures can be compared quantitatively via stable categorical persistence functions, and possibly optimized by defining bottleneck distance-based loss functions. The inference of states also plays a central role in reinforcement learning. In this context, agents--whose action policy is typically computed by deep neural networks--could be compared by considering the weighted graph defined by considering the activation of the network at each state-action pair.

We hope that this work paves the road to new applications of the persistence paradigm in various fields. We list a few possible developments that are currently being developed by our team and hopefully by other researchers.

So far we have just considered $\mathbb{R}$ as a parameter for filtrations, but there has been much progress in the study of filtering functions with $\mathbb{R}^k$ \cite{FrMu99,CaZo07,CaDiFe10} or even $\mathbb{S}^1$ \cite{BuDe13} as a range, and of spaces parametrized by a lattice \cite{CoSk13}. The definition of persistence functions should be extended to these settings.

Persistence diagrams are but a shadow of much more general and powerful tools: persistence modules and further \cite{BuSc14,Les15,deMu*17}, on which the {\it interleaving distance} plays a central role. It is necessary to connect the ideas of the present paper to that research domain.

\section*{Acknowledgments}
We are indebted to Diego Alberici, Emanuele Mingione, Pierluigi Contucci (whose research originated the present one), Luca Moci, Fabrizio Caselli and Patrizio Frosini for many fruitful discussions. MGB and PV wish to thank Champalimaud Research and Zachary F. Mainen for funding their work. Article written within the activity of INdAM-GNSAGA.

\bibliographystyle{abbrv}
\bibliography{bibliography}
\end{document}